\documentclass[a4paper,10pt]{scrartcl}

\usepackage[fleqn]{amsmath}
\usepackage{amssymb}
\usepackage{amsthm}
\usepackage{stmaryrd}
\usepackage{bbm}
\usepackage[colorlinks]{hyperref}
\usepackage{mathtools}
\usepackage[numbers]{natbib}

\mathtoolsset{showonlyrefs}

\newcommand{\ddx}[1][1]{\ifnum#1=1 \frac{d}{dx} \else \frac{d^{#1}}{dx^{#1}} \fi}
\newcommand{\ddy}[1][1]{\ifnum#1=1 \frac{d}{dy} \else \frac{d^{#1}}{dy^{#1}} \fi}
\newcommand{\ddt}[1][1]{\ifnum#1=1 \frac{d}{dt} \else \frac{d^{#1}}{dt^{#1}} \fi}

\usepackage{todonotes} 

\newenvironment{remark}[1][Remark]{\begin{trivlist}
\item[\hskip \labelsep {\bfseries #1}]}{\end{trivlist}}

\theoremstyle{plain}
\newtheorem{theorem}{Theorem}[section]  
\newtheorem{proposition}[theorem]{Proposition}
\newtheorem{corollary}[theorem]{Corollary}
\newtheorem{lemma}[theorem]{Lemma}

\theoremstyle{definition}

\theoremstyle{remark}

\numberwithin{equation}{section}

\title{On projections of the supercritical contact process: uniform mixing and cutoff phenomenon} 
\author{Stein Andreas Bethuelsen \footnote{From 01.12.2018 at the Department of Mathematics and Physics, University of Stavanger, Norway.}}

\begin{document}
\maketitle

\abstract{We consider the contact process on a countable-infinite and connected graph of bounded degree. For this process started from the upper invariant measure, we prove certain uniform mixing properties under the assumption that the infection parameter is sufficiently large. In particular, we show that the projection of such a process onto a finite subset forms a process which is $\phi$-mixing. The proof of this is based on large deviation estimates for the spread of an infection and general correlation inequalities. 
In the special case of the contact process on $\mathbb{Z}^d$, $d\geq1$, we furthermore prove the cutoff phenomenon,  valid in the entire supercritical regime. } 


\section{Introduction}

\subsection{The contact process}\label{sec CP}

The contact process is a one-parameter family of interacting particle systems, introduced by Harris in  \cite{HarrisCP1974} as a toy model for the spread of an infection in a population. 
In these processes an individual of the population is either \emph{healthy} (represented by assigning it the label $0$) or \emph{infected} (represented by assigning it the label $1$). With time, an individual recovers (or becomes healthy) at rate $1$, irrespectively of the state of the other individuals. Additionally, an individual becomes infected at a rate $\lambda$ times the number of infected individuals in its neighbourhood, where $\lambda \in (0,\infty)$ is the parameter of the model.  

More formally, the contact process  with parameter $\lambda\in(0,\infty)$, denoted here by $(\eta_t)_{t\geq0}$, is a  continuous-time Markov process on $\Omega:=\{0,1\}^V$, where $V$ is the vertex set of a graph $G=(V,E)$ representing the network on which the individuals live. For the process to be well defined, and to avoid certain technicalities, we assume throughout this paper that the graph $G$ is countable, connected and of bounded degree. We denote by $dist(x,y)$ the graph distance in $G$ between any $x,y \in V$. Then, letting $\mathcal{C}(\mathbb{R})$ denote the set of bounded and continuous functions $f \colon \Omega \rightarrow \mathbb{R}$, the contact process can be specified by its generator $L_{\lambda} \colon \mathcal{C}(\mathbb{R}) \mapsto \mathcal{C}(\mathbb{R})$, where, for $\omega \in \Omega$
\begin{equation}\label{eq contact generator}
L_{\lambda} f(\omega) := \sum_{\substack{x \in V \\ \omega(x)=1}} \left( [f(\omega^{x \leftarrow 0})-f(\omega)]+\lambda \sum_{\substack{y \in V \\ dist(y,x)=1}} [f(\omega^{y \leftarrow 1})-f(\omega)] \right).
\end{equation} 
Here, for $z \in V$ and $i \in \{0,1\}$, we denote by $\omega^{z \leftarrow i}$ the configuration where $\omega^{z \leftarrow i}(z)=i$ and $\omega^{z \leftarrow i}(x)=\omega(x)$ for any $x\neq z$. 

There is also a very useful definition of the contact process via coupling, known as the \emph{graphical construction}, which we recall in Section \ref{sec GC}. As references to interacting particle systems in general and the contact process in particular we mention the books by Liggett, \cite{LiggettIPS1985} and \cite{LiggettSIS1999}.

As can be seen immediately from the definition in \eqref{eq contact generator}, the configuration $\underbar{0} 
\in \{0,1\}^V$ where $\underbar{0}(x)=0$ for all $x\in V$, is an absorbing state. That is, if at a certain time $T$ all individuals are healthy then they remain healthy for all future times $t \geq T$. On the other hand, as is well known, the contact process having initially all individuals infected converges weakly towards a stationary (or invariant) distribution, called the \emph{upper invariant measure} and denoted here by $\bar{\nu}_{\lambda}$. In the sequel,  we denote by the \emph{upper stationary contact process} the process $\bar{\eta}=(\bar{\eta}_t)_{t \in \mathbb{R}}$ obtained as the stationary extension of the contact process under $\bar{\nu}_{\lambda}$. Further, we let $\bar{\mathbb{P}}_{\lambda}$ denote the corresponding path-space measure on the set $D_{\Omega}(\mathbb{R})$ of c\'adl\'ag functions on $\mathbb{R}$ taking values in $\Omega$ and by $\mathcal{F}$ its $\sigma$-algebra.

An important feature of the contact process is that it has a \emph{phase transition}. That is, for $G$ countable-infinite, there exists a unique (and critical) parameter value $\lambda_c=\lambda_c(G) \in (0,\infty)$ defined through the following properties. For all $\lambda< \lambda_c$ we have that $\bar{\nu}_{\lambda} = \delta_{\underbar{0}}$, where $\delta_{\underbar{0}}$ denotes the measure concentrating on $\underbar{0}$, and for all $\lambda>\lambda_c$ it holds that $\bar{\nu}_{\lambda} \neq \delta_{\underbar{0}}$. 

The contact process is said to be  \emph{supercritical} in the regime $\lambda> \lambda_c$. This also corresponds to the notion \emph{weak survival} in \cite{LiggettSIS1999} and is equivalent to having that, with positive probability, the process with initially only one infected individual never reaches the state $\underbar{0}$. Other critical parameter values have been considered in the literature, most notably the one defined through the notion of \emph{strong survival}. That is, the contact process survives strongly if, with positive probability, when initiated with only one infected individual, this individual will become (re)-infected from its neighbours  infinitely often. 

Clearly, strong survival implies weak survival. For the contact process on the integer lattices $\mathbb{Z}^d$, $d\geq1$, and many other graphs, these two critical parameter values coincide. Interestingly, for a large class of graphs, such as all d-regular trees with degree $d\geq3$, there is a regime of infection parameters where the contact process survives weakly, but not strongly. 

In Section \ref{sec discussion} we introduce yet another critical parameter value for the contact process and discuss its relation to that of weak and strong survival. Besides this, for the rest of the paper, we mainly restrict our analysis to cases where $\lambda$ is large and the contact process is well within the strong survival regime. Sufficient for this is the assumption that $\lambda> \lambda_c(\mathbb{N})$. This assumption allow us to state our theorems without posing strong restrictions on the underlying graph. 

To this end, it should be noted that $\lambda_c(\mathbb{N})= \lambda_c(\mathbb{Z})$, as proven in \cite{DurretGriffeathCP1983}. We choose to write $\lambda>\lambda_c(\mathbb{N})$ in order to emphasise that in the proofs of the following statements, we make use of the fact that, for any connected and countable-infinite graph $G=(V,E)$, to each vertex $x \in V$ there is an embedded subgraph  isomorphic to the graph $\mathbb{N}$ with $x$ as the ``origin''.

\subsection{Projections of the contact process onto finite subsets}
Our main objects of interest in this paper are projections of the contact process onto finite subsets of the underlying graph.  That is, we only observe the evolution of a partial and finite selection of all the individuals in the population. To make this precise, denote by $\mathcal{S}$ the set of all finite subsets of $V$. 
Then, given $\Delta \in \mathcal{S}$, we consider the process $\xi^{(\Delta)} =(\xi_t^{(\Delta)})_{t\in \mathbb{R}}$ on $\Omega_{\Delta}:=\{0,1\}^{\Delta}$ obtained by projecting the upper stationary contact process onto $\Delta$, that is, 
\begin{equation}\label{eq projection process}
\xi_t^{(\Delta)}(x) := \bar{\eta}_t(x), \quad x \in \Delta, t \in \mathbb{R}.
\end{equation} 
Thus, the process $\xi^{(\Delta)}$  
is a function of the Markov process $\bar{\eta}$ and as such it is an example of a \emph{hidden Markov process}.  By the stationarity of $\bar{\eta}$ it follows that also the  process $\xi^{(\Delta)}$ is a stationary process. However,  in contrast to $\bar{\eta}$, we note that $\xi^{(\Delta)}$ is not Markovian. 

Our motivation for studying the processes $(\xi^{(\Delta)})_{\Delta \in \mathcal{S}}$ is manifold. Firstly, they are  fairly natural processes to study in the view of  the contact process as a population model. Indeed, populations are typically very large and often it is in practice impossible to study the evolution of the entire population. One therefore is essentially forced  to consider observations within partial subsets of the population only.

Secondly, there has recently been a boost of  work on the contact process evolving on finite graphs. Although the contact process on finite graphs can neither survive strongly nor weakly, the time until it reaches its stationary state may depend crucially on the starting configuration and the parameter value. In fact, metastable behaviour have recently been shown to hold for the contact process in great generality concerning the graph structure as soon as the infection parameter is large enough, see e.g.\  \cite{MountfordMourratValesinYao2016} and \cite{ShapiraValesin2017}. 
Motivated by the progress seen for the contact process on general finite graphs, in this paper we study the asymptotic properties of $(\xi^{(\Delta)})_{\Delta \in \mathcal{S}}$ for the contact process on general countable-infinite graphs, for which previous works have mainly been devoted  its study on the particular graphs $\mathbb{Z}^d$ and $\mathbb{T}^d$, $d\geq 1$. (See, however, \cite{SalzanoSchonmann1997} and \cite{SalzanoSchonmann1999} for two notable exceptions).

A third motivation has been to further highlight the potentials of the methods developed in  \cite{BergBethuelsenCP2018} and \cite{LiggettSteifSD2006}. These techniques rely strongly on the so-called \emph{downward FKG} property as shown to hold for the contact process in \cite{BergHaggstromKahn2006}. (See Section \ref{sec prel mono} for a precise definition). We have focused here on proving strong (uniform) mixing properties for $\xi^{(\Delta)}$,  extending upon the mixing properties proven in \cite{BergBethuelsenCP2018}. These mixing properties are very useful to the mathematical analysis of the contact process, of which we provide several basic examples. 

For further discussions on the current work we refer to the next section, where our main results are presented, and to Section \ref{sec discussion}, where we also phrase some in our opinion intriguing open questions.

\subsection{Outline of the paper}

In Subsection \ref{sec mixing} we present our main results on the mixing properties of $\xi^{(\Delta)}$ and in Subsection \ref{sec applications} we discuss certain applications of these mixing properties. In Section \ref{sec prel} we introduce the graphical construction of the contact process and discuss some of its key properties. Proofs  of our main results are deferred to Section \ref{sec proofs}. We end with a short discussion in Section \ref{sec discussion}.

\section{Main results and applications}
Throughout this section we let $\bar{\eta}$ denote the upper stationary contact process on a \emph{countable-infinite} and \emph{connected} graph $G=(V,E)$ of \emph{bounded degree}.  Given $\Delta \in \mathcal{S}$, we focus our study on the mixing properties of the projected process $\xi^{(\Delta)}$. 

\subsection{Main results}\label{sec mixing}
Projections of the contact process have previously been considered in the literature by several (e.g.\  \cite{GalvesMarinelliOlvieriCP1989}, \cite{GoldsteinWiroonsriCP2018},  \cite{SchonmannCLT1986} and \cite{SimonisCP1998}), but to our knowledge only on the particular graphs $\mathbb{Z}^d$, $d\geq1$. 
In \cite{GoldsteinWiroonsriCP2018} and \cite{SchonmannCLT1986} two key properties of the contact process on $\mathbb{Z}$ were used in order to derive various statistics, in particular the central limit theorem, for $\xi^{(\Delta)}$ in the supercritical regime. On the one hand, the contact process is monotone (in the sense of positive association; see Section \ref{sec prel mono} for a precise definition) and on the other hand it has fast decay of temporal correlations. 

The first property extends immediately to the contact process considered on any other graph.  We next present a generalisation of the latter property. 
For $t \in \mathbb{R}$, we write  $\mathcal{F}_{\geq t}^{\Delta}$ and $\mathcal{F}_{\leq t}^{\Delta}$ for the sub-$\sigma$-algebras of $\mathcal{F}$ generated by events in $\Delta \times [t,\infty)$ and $\Delta \times (-\infty,t)$, respectively. 

\begin{proposition}\label{prop covariance} 
If $\lambda>\lambda_c(\mathbb{N})$, then for any $\Delta \in \mathcal{S}$, the projected process $\xi^{(\Delta)}$ is $\alpha$-mixing (or strong mixing), that is, 
\begin{equation}\label{eq covariance}
\lim_{t \rightarrow \infty} \sup_{A_0 \in  \mathcal{F}_{\leq 0}^{\Delta}} \sup_{B_t \in \mathcal{F}_{\geq t}^{\Delta}}
 \left| \bar{\mathbb{P}}_{\lambda} \left( B_t \cap A_0 \right) - \bar{\mathbb{P}}_{\lambda}\left(A_0\right) \bar{\mathbb{P}}_{\lambda}\left(B_t\right) \right| = 0.
\end{equation}
Moreover, there are constants $C,c \in (0,\infty)$ depending on $\lambda$ only, such that, for fixed $t\in (0,\infty)$, the lefthand side of \eqref{eq covariance} decays as $C|\Delta|e^{-ct}$.
\end{proposition}

The main purpose of this subsection is to present a generalisation of the mixing property in \eqref{eq covariance} which, for any event $A_0 \in \mathcal{F}_{\leq 0}^{\Delta}$, compares on $\mathcal{F}_{\geq t}^{\Delta}$ the conditional measure $\bar{\mathbb{P}}_{\lambda}(\cdot \mid A)$ with its unconditional counterpart $\bar{\mathbb{P}}_{\lambda}(\cdot)$, asymptotically for large times $t$. For this, we introduce the (non-increasing) function $d_{\Delta} \colon [0,\infty) \rightarrow [0,1]$ given by
\begin{equation}\label{eq phi mixing}
d_{\Delta}(t):= \sup_{B_t \in \mathcal{F}_{\geq t}^{\Delta}}\sup_{\substack{A_0 \in \mathcal{F}_{\leq 0}^{\Delta}\\ \bar{\mathbb{P}}_{\lambda}(A_0)>0 }} \left| \bar{\mathbb{P}}_{\lambda}(B_t \mid A_0) - \bar{\mathbb{P}}_{\lambda}(B_t) \right|, \end{equation} 
which is the corresponding \emph{mixing time} of $\xi^{(\Delta)}$ with respect to the total variation distance.  
One of our main contributions in this paper is the following theorem.

\begin{theorem}\label{thm phi mixing}
If $\lambda>\lambda_c(\mathbb{N})$, then for any $\Delta \in \mathcal{S}$, the projected process $\xi^{(\Delta)}$ is $\phi$-mixing, that is, $\lim_{t\rightarrow \infty} d_{\Delta}(t) =0$.
\end{theorem}

A $\phi$-mixing process is also $\alpha$-mixing, but the opposite is in general not true. In this sense, Theorem \ref{thm phi mixing} is a strengthening of  Proposition \ref{prop covariance}. As we exemplify in the next subsection, both $\alpha$-mixing and $\phi$-mixing are strong forms of mixing having several applications to the study of the long term behaviour of $\xi^{(\Delta)}$, and in general also to $(\eta_t)_{t\geq0}$. For a general account on mixing properties for processes, we refer to \cite{BradleyStrongMixing2005}.

The proof of Theorem \ref{thm phi mixing} is based on certain modifications of a stochastic domination argument in \cite{BergBethuelsenCP2018} together with the following property for the contact process conditioned on survival. In order to state this property precisely let, for $\Lambda \subset V$,
\begin{equation}
\tau^{\Lambda} := \inf \left\{ t \geq 0 \colon \eta_t^{\Lambda} \equiv \underbar{0} \right\}
\end{equation}
be the first time $(\eta_t^{\Lambda})_{t\geq0}$ reaches the state $\underbar{0}$, where $(\eta_t^{\Lambda})_{t\geq0}$ denotes the contact process with the initial state satisfying $\eta_0^{\Delta}(x) =1$ if and only if $x\in \Lambda$. We denote the corresponding path-space measure by $\mathbb{P}_{\lambda}^{\Lambda}$.  If $\Lambda=\{y\}$ for some $y\in V$, we for simplicity write $\tau^{y}$ and $\mathbb{P}_{\Lambda}^y$. Further, for $y\in V$ and $\theta \in (0,\infty)$, let
\begin{equation}
C_{\theta}(y):= \left\{(x,s)\in V \times [0,\infty) \colon dist(y,x) \leq t\theta \right\}
\end{equation} 
be the ``forward cone'' of inclination $\theta$ and center at $y$. For $t\in [0,\infty)$, let $C_{\theta,t}(y) = C_{\theta}(y) \cap V \times [t,\infty)$ and denote by $\mathcal{F}_{\theta,t}(y)$ the $\sigma$-algebra generated by the events in $C_{\theta,t}(y)$.

\begin{theorem}\label{thm local shape theorem}
Let $\lambda>\lambda_c(\mathbb{N})$. Then there are constants $\theta, C,c \in (0,\infty) $ such that, for all $y \in V$,
\begin{align}
\sup_{B_t \in \mathcal{F}_{\theta,t}(y)} \left| \mathbb{P}_{\lambda}^y( B_t \mid \tau^{y} = \infty) - \bar{\mathbb{P}}_{\lambda}(B_t) \right| \leq Ce^{-ct}, \quad \forall t\geq0.
\end{align}
\end{theorem}

The proof of Theorem \ref{thm local shape theorem} goes by a modification of an argument in \cite{DurrettGriffeathCPhighDim1982}, who considered the contact process on $\mathbb{Z}^d$, to more general graphs. 

In general, we believe that the mixing time in Theorem \ref{thm phi mixing} decays exponentially in $t$, similarly to the bound obtained in Proposition \ref{prop covariance}.   Further, in many cases the assumption that $\lambda>\lambda_c(\mathbb{N})$ is clearly too strong. We end this subsection with a result which shows that this indeed is the case for the contact process on $\mathbb{Z}^d$, $d\geq1$. For this, for $n \in \mathbb{N}$, let $\Delta_n := [-n,n]^d$ and denote by $\xi^{(n)}$ the projection of $\bar{\eta}$ onto $\Delta_n$. 
Further, denote by
\begin{equation}\label{eq t(x)}
t(x) := \inf \{ t \geq 0 \colon \eta_t^o(x)=1 \}, \quad x \in \mathbb{Z}^d,
\end{equation}
the first time the contact process with only the individual at the origin initially infected spreads its infection to the individual at $x$. 
From the proof of the \emph{shape theorem} for the contact process (see e.g.\ \cite{GaretMarchandShape2012}, Section 5), it follows that, $\mathbb{P}_{\lambda}^o(\cdot \mid \tau^o=\infty)$-a.s.,
\begin{equation}\label{eq beta}
\lim_{n \rightarrow \infty} \frac{t(n \cdot e_1)}{n} :=\beta
\end{equation}
 for a constant $\beta=\beta(\lambda)$ representing the asymptotic linear speed for the spread an infection in the direction of $e_1:=(1,0, \dots, 0)$, and which is strictly positive whenever $\lambda>\lambda_c(\mathbb{Z}^d)$. 

\begin{theorem}\label{thm phi mixing Zd}
Let $\lambda>\lambda_c(\mathbb{Z}^d)$. Then, for any $\epsilon>0$, there are constants $C,c \in (0,\infty)$ such that for all $n \in \mathbb{N}$,
\begin{equation}\label{eq phi mixing Zd upper}
d_{\Delta_n}(t) \leq C |\Delta_n| e^{-ct}, \quad \forall \: t \geq \beta n (1+\epsilon).
\end{equation}
Furthermore, we have that
\begin{equation}
\begin{aligned}\label{eq cutoff}
\lim_{n \rightarrow \infty} d_{\Delta_n}\left( \beta n (1+\epsilon)\right) =0 \quad \text{ and } \quad \lim_{n \rightarrow \infty} d_{\Delta_n}\left( \beta n (1-\epsilon)\right) =1.
\end{aligned}
\end{equation}
\end{theorem}

In terms of mixing times, as typically studied for Markov chains (see e.g.\ \cite{LevinPeresWilmer2017}),  the second part of Theorem \ref{thm phi mixing Zd} says that the processes $(\xi^{(n)})_{n\in \mathbb{N}}$ has (asymptotically) a cutoff at time $\tau_n:=\beta n$, called the mixing time of $(\xi^{(n)})_{n\in \mathbb{N}}$. 
As we discuss more thoroughly in Section \ref{sec thm phi mixing Zd} (see also Section \ref{sec discussion}), we furthermore believe that the estimate on the mixing time in \eqref{eq phi mixing Zd upper} hold under minimal assumptions on the graph structure as soon as $\lambda>\lambda_c(\mathbb{N})$. 

\subsection{Applications}\label{sec applications}

As a first application of the mixing properties stated in the previous subsection, we  focus on large deviation estimates for $\xi^{(\Delta)}$. For $f \colon \Omega_{\Delta} \rightarrow \mathbb{R}$ bounded, denote by 
\begin{equation}Z_t^{(\Delta)}(f) := \int_{0}^tf(\xi_t^{(\Delta)}) dt, \quad t >0.\end{equation}
By Theorem \ref{thm phi mixing} it follows that also $(Z_t^{(\Delta)}(f))_{t\in \mathbb{R}}$ is $\phi$-mixing. Large deviation estimates for $\phi$-mixing sequences have been proven in \cite{SchonmannMixing1989}. 
By an immediate extension of the proof in \cite{SchonmannMixing1989} to continuous-time processes, we obtain our first application of Theorem \ref{thm phi mixing}, as stated next.

\begin{corollary}\label{cor LDP}
Let $\lambda > \lambda_c(\mathbb{N})$. Then, for any $\epsilon>0$ there are constants $C,c>0$ such that 
\begin{equation}
\bar{\mathbb{P}}_{\lambda} \left( Z_t^{(\Delta)}(f)/t \notin  (\bar{\nu}_{\lambda}(f)- \epsilon, \bar{\nu}_{\lambda}(f)+ \epsilon) \right) \leq Ce^{-ct},\quad \forall \: t\geq0.
\end{equation}
\end{corollary} 

By combining the large deviation estimates in Corollary \ref{cor LDP} with standard subadditivity arguments as in \cite{LebowitzSchonmannLDP1988}, Section $3$, we furthermore obtain a large deviation principle. To state this precisely, consider the partial ordering on $\Omega_{\Delta}$ given by $\sigma \leq \omega$ if and only if $\sigma(x) \leq \omega(x)$ for all $x\in \Delta$. We say that a function $f \colon \Omega_{\Delta} \rightarrow \mathbb{R}$ is \emph{increasing} if $f(\sigma) \leq f(\omega)$ whenever $\sigma \leq \omega$.

\begin{theorem}\label{thm LDP}
 Let $\lambda > \lambda_c(\mathbb{N})$ and $f \colon  \Omega_{\Delta} \rightarrow \mathbb{R}$ bounded and increasing. 
 Then there exists $\psi \colon \mathbb{R} \rightarrow \{-\infty\} \cup (-\infty,0]$ concave such that $\psi(x) < 0$ for all $x \neq \bar{\nu}_{\lambda}(f)$ and, for all $a<b$,
\begin{equation}
\lim_{t \rightarrow \infty} t^{-1} \log \bar{\mathbb{P}}_{\lambda} \left(\frac{Z_t^{\Delta}(f)}{t} \in [a,b] \right) =  \sup_{a \leq x \leq b} \psi(x).
\end{equation}
\end{theorem}

We remark that the large deviation estimates in Corollary \ref{cor LDP} hold without the assumption that $f$ is increasing, whereas this assumption is important to the analysis via subadditivity arguments in \cite{LebowitzSchonmannLDP1988}.

Other asymptotic properties of $Z_t^{(\Delta)}(f)$, such as the central limit theorem, have earlier been considered for the contact process on $\mathbb{Z}$; see \cite{SchonmannCLT1986} and \cite{GoldsteinWiroonsriCP2018}, Section 2.4.  These  works extends to the contact process on general countable-infinite graphs of bounded degree.

\begin{corollary}\label{prop CLT}
Let $\lambda>\lambda_c(\mathbb{N})$ and consider $f \colon \Omega_{\Delta} \rightarrow \mathbb{R}$ bounded. Then there exists $0\leq \sigma_f^2 < \infty$ such that 
\begin{equation}\label{eq CLT}
t^{1/2} \left[t^{-1} Z_t^{(\Delta)}(f) - \bar{\nu}_{\lambda}(f) \right] \rightarrow_L N(0,\sigma_f^2)\quad \text{as } t \rightarrow \infty,
\end{equation}
where $\rightarrow_L$ denotes convergence in law.
\end{corollary}

The proof of Corollary \ref{prop CLT} follows by Proposition \ref{prop covariance} applied to the proof strategy outlined in \cite{SchonmannCLT1986}. In fact, Corollary \ref{prop CLT} is the generalisation of  Lemma 1 in  \cite{SchonmannCLT1986} to more general graphs and the proof of this lemma goes through after only minor modifications, replacing the last estimate with that of Proposition \ref{prop covariance}. 

We remark that Corollary \ref{prop CLT} can be further strengthened to the immediate extension of Theorem 1 in \cite{SchonmannCLT1986}. For this, the estimates in Theorem \ref{thm local shape theorem} are needed. Moreover, the proofs in \cite{GoldsteinWiroonsriCP2018} work in our generality as well and yield more quantitive bounds on the rate of convergence in \eqref{eq CLT} whenever $f$ is increasing (for which $\sigma_f^2>0$; see \cite{SchonmannCLT1986}).

As a last application, we mention the complete convergence theorem, which is a consequence of  Theorem \ref{thm local shape theorem}.
\begin{corollary}\label{cor CCT}
Let $\lambda>\lambda_c(\mathbb{N})$. Then, for any $\Lambda \subset V$, 
\begin{align}
\mathbb{P}_{\lambda}^{\Lambda} (\eta_t \in \cdot)
\implies \mathbb{P}_{\lambda}^{\Lambda} (\tau^{\Lambda} <\infty) \delta_{\underbar{0}}(\cdot) + \mathbb{P}_{\lambda}^{\Lambda}(\tau^{\Lambda}=\infty) \bar{\nu}_{\lambda}(\cdot), \text{ as } t \rightarrow \infty.
\end{align}
\end{corollary}

For a proof of Corollary \ref{cor CCT} based on Theorem \ref{thm local shape theorem} we refer to  \cite{LiggettIPS1985}, p.\ 284-287, where the proof for the case of the contact process on $\mathbb{Z}$ is given in much detail. The proof strategy therein applies after minor notational modifications only also to the contact process on general graphs as soon as the estimate in Theorem \ref{thm local shape theorem} is at hand. To the best of our knowledge, in this generality, the complete convergence theorem was first proven in \cite{SalzanoSchonmann1997}, Theorem $5$, by an elegant argument, however, not relying on any explicit mixing property. 

\section{Preliminaries}\label{sec prel}
In this section we recall the graphical construction of the contact process and some of its basic consequences. For a more thorough  description we refer to \cite{LiggettSIS1999}, Chapter 1.

\subsection{The graphical construction}\label{sec GC}

Let $G=(V,E)$ be a connected graph having bounded degree and fix $\lambda \in (0,\infty)$. 
 Assign  Poisson processes $N_x$ on $\mathbb{R}$ of rate $1$ to each individual $x \in V$ and Poisson processes $N_{(x,y)}$ on $\mathbb{R}$ of rate $\lambda$ to each ordered pair of individuals satisfying $dist(x,y)=1$. All these processes are taken independent of each other and together yield the following space-time picture on $V \times \mathbb{R}$.  
 
 To each event $t$ of $N_x$ draw a \emph{star} $\star$ at $(x,t)$ and to each event $s$ of $N_{(x,y)}$ draw an arrow $\rightarrow$ from $(x,s)$ to $(y,s)$. 
 The stars resembles the event that an individual becomes healthy and arrows resembles the possible spreading of an infection. From these Poisson processes we can construct the contact process using terminology from percolation theory. For $x,y \in V$ and $ s \leq t$, we say that $(x,s)$ is connected to $(y,t)$ by an \emph{active} path, written $(x,s) \rightarrow (y,t)$, if and only if there exists a directed path in $V \times \mathbb{R}$ starting at $(x,s)$, ending at $(y,t)$ and going either forwards in time without hitting crosses or ``sideways'' following arrows in the direction of the prescribed direction. Otherwise we write $(x,s) \nrightarrow (y,t)$. In general, for $\Lambda, \Delta \subset V \times \mathbb{R}$, we write $\Delta \rightarrow \Lambda$ ($\Delta \nrightarrow \Lambda$) if there is a (there is no) active path from $\Lambda$ to $\Delta$.  
 
 Now, for $\Lambda \subset V$ and $s\in \mathbb{R}$, the contact process $(\eta_t^{(\Lambda,s)})_{t\geq s}$ on $G$ with infection parameter $\lambda$  and satisfying $\eta_s(x) = 1 $ if and only if $x \in \Lambda$ can be defined by 
 \begin{equation}\label{def cp gp}
 \eta_t^{(\Lambda,s)}(x) = \left\{\begin{array}{cc}1, &  \text{ if } (\Lambda,s) \rightarrow (x,t); \\0, & \text{otherwise}.\end{array}\right.
 \end{equation}
 If $s=0$ we simply write $(\eta_t^{\Lambda})_{t\geq0}$ for the above defined process. It is well known that the process defined by \eqref{def cp gp}, known as the  \emph{graphical construction} of the contact process, has the same distribution as the process defined via the generator in \eqref{eq contact generator}. 

In many of the proofs in the following section it will be useful to consider the contact process evolving on certain subgraphs of $G$. For a graph $G_1=(V_1,E_1)$ with $V_1 \subset V$ and $E_1\subset E$, we write $(^{G_1}\eta_t^{(\Lambda,s)})_{t\geq s}$ for the contact process evolving on $\{0,1\}^{V_1}$ where in the graphical construction we only consider the events $(N_x)_{x\in V_1}$ and $(N_{(x,y)})_{(x,y)\in E_1}$. Further,  we introduce the notation
\begin{equation}
^{G_1}\tau^{(\Lambda,s)} := \inf \left\{ t\geq s \colon \Lambda \times \{s\} \nrightarrow (y,t) \text{ in } G_1 \text{ for any } y \in V_1\right\},
\end{equation}
which is the time until the contact process started at time $s$ with only the individuals inside $\Lambda$ initially infected reaches the state $\underbar{0}$ when restricting the graphical construction to $G_1$.

\subsection{Consequences of the graphical construction}\label{sec prel mono}

The graphical construction yields an elegant description of the upper stationary contact process:
 \begin{equation}\label{def uscp gp}
\bar{\eta}_t(x) := \left\{\begin{array}{cc}1, &  \text{ if }V \times \{-\infty \} \rightarrow (x,t); \\0, & \text{otherwise}.\end{array}\right.
\end{equation}
In \eqref{def uscp gp} and later, we denote by  $V \times \{-\infty \} \rightarrow (x,t)$ the event that  for all $s\leq t$ there exists an active path from $V \times \{s\}$ to $(x,t)$.

A further advantage of the graphical construction is that it yields a natural coupling of contact processes with different starting configurations  or infection parameters, denoted in the following by $\widehat{\mathbb{P}}_{\lambda}$. Moreover, from the construction it is easy to see that the contact process is \emph{monotone}, in the sense that, if $\Delta \subset \Lambda$, then $\widehat{\mathbb{P}}_{\lambda}$-a.s., 
$\eta_t^{\Delta}(x) \leq \eta_t^{\Lambda}(x)$ for all $x \in V$ and all $t \in [0,\infty)$. In fact, the contact process is \emph{positively associated}. In particular, for any two increasing events $A,B \in \mathcal{F}$, we have that 
\begin{equation}\label{eq positive associated}
\bar{\mathbb{P}}_{\lambda}(A \cap B) \geq \bar{\mathbb{P}}_{\lambda}(A) \bar{\mathbb{P}}_{\lambda}(B).
\end{equation}
Here, an event $A$ is said to be increasing if it has the property that when $(\omega_t) \in A$ then also $(\sigma_t) \in A$ whenever $\omega_t\leq \sigma_t$ for all $t$.  

Another very useful property is that the contact process is \emph{self-dual}, that is, for any $\Delta, \Lambda \in \mathcal{S}$, 
\begin{equation}
\mathbb{P}_{\lambda}^{\Delta} \left( \eta_t^{\Delta}(x)=1 \text{ for some } x \in \Lambda \right) = \mathbb{P}_{\lambda}^{\Lambda} \left( \eta_t^{\Lambda}(x)=1 \text{ for some } x \in \Delta \right).
\end{equation}
To see this from the graphical construction, first recall that $\{ \eta_t^{\Delta}(x)=1 \text{ for some } x \in \Lambda\} = \{ (\Delta,0) \rightarrow (\Lambda,t) \}$. This event again equals $\{ (\Delta,0) \leftarrow (\Lambda,t) \}$, where we write $\{ (\Delta,0) \leftarrow (\Lambda,t) \}$ if  there exists a directed \emph{backwards}-path in $V \times \mathbb{R}$ starting at $\Delta$, ending at $\Lambda$ and going either \emph{backwards} in time without hitting crosses or ``sideways'' following arrows in the \emph{opposite} direction of the prescribed direction. The latter event defines the \emph{dual process} $(\tilde{\eta}_t^{(\Lambda,s)})_{t \geq 0}$ given by
\begin{equation}
\tilde{\eta}_t^{(\Lambda,s)}(x) := \left\{\begin{array}{cc}1, &  \text{ if }(x,s-t) \leftarrow (\Lambda,s); \\0, & \text{otherwise},\end{array}\right.
\end{equation}
It is clear by the construction that $(\tilde{\eta}_t^{(\Lambda,s)})$ has the same distribution as $(\eta_t^{(\Lambda,0)})$ and hence that the contact process is self-dual.

The last property of the graphical construction we want to discuss is that of \emph{downward FKG} (abbreviated by dFKG), which is an even stronger correlation inequality than positive association introduced above. In particular, for each $\Lambda \in \mathcal{S}$ and any interval $[s_1,s_2] \in \mathbb{R}$, it says that
\begin{equation}\label{eq dFKG SD2}
\bar{\mathbb{P}}_{\lambda} \left(\cdot \mid \bar{\eta}_t(x) =0 \: \forall \: (x,t) \in \Lambda \times [s_1,s_2] \right)
\end{equation}
is positively associated.  This property was first proven in \cite{BergHaggstromKahn2006} for the upper invariant measures and in \cite{BergBethuelsenCP2018} it was noted that the proof in  \cite{BergHaggstromKahn2006} extends, again by use of the graphical construction, to the upper stationary contact process. 

The dFKG property is very useful when considering conditional measures since it implies that
\begin{equation}\label{eq dFKG SD}
\bar{\mathbb{P}}_{\lambda} \left( B \mid \bar{\eta}_t(x) =0 \: \forall \: (x,t) \in \Lambda \times [s_1,s_2] \right) \leq \bar{\mathbb{P}}_{\lambda}\left(B \mid A\right),
\end{equation}
where $B \in \mathcal{F}$ is any increasing event depending only on the contact process outside $\Lambda \times [s_1,s_2]$ and $A \in \mathcal{F}$ is any event depending only on the contact process within $\Lambda \times [s_1,s_2]$. In other words, the measure $\bar{\mathbb{P}}_{\lambda}(\cdot \mid A)$ \emph{stochastically dominates} the measure $\bar{\mathbb{P}}_{\lambda}( \cdot \mid \bar{\eta}_t(x) =0 \: \forall \: (x,t) \in \Lambda \times [s_1,s_2] )$. For a proof of this fact, see \cite{BergBethuelsenCP2018}.

To this end, we should emphasise that positive association may fail for other conditional measures of the contact process than that on the lefthand side of  \eqref{eq dFKG SD2}. 
For instance, \cite{LiggettCPnFKG1994} proved that this is the case for the upper invariant measure of the contact process on $\mathbb{Z}$ conditioned on having an infected individual at the origin. On the other hand, the dFKG property is known hold for a wide range of percolation-like models as shown in  \cite{BergHaggstromKahn2006}. (See also \cite{LiggettCPcor2006} and \cite{BergHaggstromKahnKonno2006} for further discussions on correlation inequalities for the contact process).

\section{Proofs}\label{sec proofs}

In this section we present the proofs of the statements given in Section \ref{sec mixing}. Detailed proofs of the statements in Section \ref{sec applications} are left to the reader. As standard, throughout this section, the numbers $C,c\in (0,\infty)$ represent constants whose value might change from line to line in the calculations, but which remain strictly bounded away from $0$ and $\infty$.

\subsection{Proof of Proposition \ref{prop covariance}}

Let $G=(V,E)$ be a countable-infinite and connected graph of degree bounded by $D \in \mathbb{N}$. Further, to each individual $y \in V$ we assign a self-avoiding nearest neighbour path $\gamma^{(y)}= (\gamma_0, \gamma_1,\gamma_2, \dots)$ in $V$. That is, the path $\gamma^{(y)}$ satisfies $\gamma_0=y$, $dist(\gamma_i, \gamma_{i+1})=1$ for each $i \geq0$, and $\gamma_i \neq \gamma_j$ for $i \neq j$.  We denote by $\Gamma^{(y)}$ the corresponding subgraph  with vertex set $V_1=\gamma^{(y)}$ and edge set $E_1:= \{ e\in E \colon e=(\gamma_{i-1},\gamma_{i}), i\in \mathbb{N} \}$. 

The following lemma is a very useful generalisation of Proposition 4 in \cite{DurrettGriffeathCPhighDim1982}, important to the proofs of Proposition \ref{prop covariance} and Theorem \ref{thm local shape theorem}.

\begin{lemma}\label{lem DG1}
Let $\lambda>\lambda_c(\mathbb{N})$. 
Then there are constants $C,c \in (0,\infty)$ such that, for all $y \in V$, 
\begin{align}\label{eq lem tau Z to V}
\mathbb{P}_{\lambda}^y \left( t <\tau^y < \infty \right) \leq Ce^{-ct}, \quad \forall \: t\geq0.
\end{align}
\end{lemma}
\begin{proof}
This follows by exactly the same argument as in the proof of Proposition 4 in \cite{DurrettGriffeathCPhighDim1982} via a so-called restart argument. For completeness, we present the key ideas of the proof. Fix $t \in (0,\infty)$ and $y\in V$. Let $v_0=0$, and define iteratively $(v_k)_{k\geq1}$ as follows. If $\eta_{v_k}^y=\emptyset$, then we set $v_{k+1}=v_k$. Otherwise, if $\eta_{v_k}^y\neq\emptyset$, let $x_k$ be a randomly chosen vertex with the property that $\eta_{v_k}^y(x_k)=1$. Let $(\eta_{v_k +t}^{(k)})$ denote the truncated contact process started at time $v_k$ with only individual $x_k$ infected and using the graphical construction only within the subgraph  $\Gamma^{(x_k)}$. Consequently,
we have that $\eta_{v_k +t}^{(k)}\leq \eta_{v_k +t}^{y}$ for all $t\geq0$ a.s. Now, set $v_{k+1} = \inf \{ s\geq v_k \colon \eta_{v_k +s}^{(k)} = \emptyset \}$ and let 
\begin{equation}
k_0 := \inf \left\{ k\geq 0 \colon v_{k}=v_{k+1} \text{ or } v_{k+1}=\infty \right\}.
\end{equation}
Thus, since $\lambda >  \lambda_c(\mathbb{N})$, on each trial time $k$ there is a positive probability that $v_k=\infty$. In particular, $k_0$ is majorised by a geometrically distributed random variable and, moreover,  we have that 
\begin{equation}
\left\{ t < \tau^y < \infty\} = \{t < v_{k_0} < \infty\right\}.
\end{equation}
Further, from the graphical construction we have that, given $k_0=k$, the times $v_j-v_{j-1}$ for $1\leq j \leq k$ are i.i.d.  Moreover, they have exponential tails since \eqref{eq lem tau Z to V} holds  for the contact process on $\Gamma^{(y)}$ uniformly in $y$. Indeed, the contact process on $\Gamma^{(y)}$ is isomorphic to the contact process on $\mathbb{N}$ for which this property was  proven in \cite{DurretGriffeathCP1983}, Theorem 5. (The paper \cite{DurretGriffeathCP1983} treated the contact process on $\mathbb{Z}$, but their methods can easily be adapted to the process on $\mathbb{N}$, see the discussion on p.\ 6 in \cite{DurretGriffeathCP1983}). By combining these two properties, we conclude the proof of the lemma.
\end{proof}

\begin{corollary}\label{cor covariance}
Let $\Delta \in \mathcal{S}$. Then there are constants $C,c \in (0,\infty)$ such that,
\begin{equation}
\widehat{\mathbb{P}}_{\lambda} \left( A_{\Delta,0,t}  \right) \leq C|\Delta|e^{-ct}, \quad \forall \: t\geq0,
\end{equation}
where $A_{\Delta,0,t} := \{ (V\times \{0\}) \leftarrow (\Delta \times [t,\infty))\} \cap \{ (V\times \{-\infty\}) \leftarrow (\Delta \times [t,\infty))\}$.
\end{corollary}

\begin{proof}
Together with the backward-paths started from $\Delta \times \{t\}$, it is sufficient to control the paths starting at the (randomly distributed) space-time points 
\begin{equation}
B_k:= \{ (x,N_{x,y}) \in \Delta \times [t+k,t+k+1) \colon dist(x,y)=1 \},\quad  k\geq 0.
\end{equation}
 Hence, by a union bound estimate, using the self-duality property and the independence structure of the graphical construction, we have that
\begin{align}
\widehat{\mathbb{P}}_{\lambda} \left( A_{\Delta,0,t}  \right) 
 &\leq Ce^{-ct}(|\Delta| + \sum_{k\geq 0}e^{-ck} \sum_{l\geq0} l  \widehat{\mathbb{P}}_{\lambda}(B_k=l) ) \leq C' |\Delta|e^{-c't}
\end{align}
for some constants $C'c,' \in (0,\infty)$, since  $ \sum_{l\geq0} l  \mathbb{P}(B_k=l) <\lambda |\Delta| D$ for any $k$, and since $\sum_{k\geq 0}e^{-ck}<\infty$.
\end{proof}

\begin{proof}[Proof of Proposition \ref{prop covariance}]
The proof follows by essentially the same argument as that for the proof of Lemma $2$ in \cite{SchonmannCLT1986}, where a slightly weaker property was shown. Indeed, \cite{SchonmannCLT1986} proved the inequality \eqref{eq covariance} for the contact process on $\mathbb{Z}$ in the particular case where the events $A\in \mathcal{F}_{\leq0}^{(\Delta}$ and $B\in \mathcal{F}_{\geq 0}^{(\Delta)}$ only  depend on the contact process at time $0$ and $t$, respectively.

For our extension, it is sufficient to consider events of the form 
\begin{align}
&A = \{ \eta_s(x) = \omega_s(x), (x,s) \in \Delta \times (-r_1,0] \};
\\&B = \{ \eta_s(x) = \omega_s(x), (x,s)\in \Delta \times [t,t+r_2) \},
\end{align}
where $r_1,r_2 \in (0,\infty)$ and $(\omega_s)_{s \in \mathbb{R}} \in D_{\Omega}(\mathbb{R})$. Then, letting 
 \begin{equation}
 \Delta_{B} := \{ (x,s) \in \Delta \times [t,t+r_2) \colon \omega_s(x)=1\},
 \end{equation}
and by replacing the definitions of $B'$ and $E'$ in the proof of Lemma $2$ in \cite{SchonmannCLT1986} by their natural generalisations, namely
\begin{align}
B' := \{ V \times \{-\infty\} \leftarrow \Delta_B \} \text{ and }
E:= \{ V \times \{0\} \leftarrow \Delta_B \},
\end{align}
we attain, by the same proof as that of Lemma $2$ in \cite{SchonmannCLT1986}, the inequality
\begin{align}
|\bar{\mathbb{P}}_{\lambda}(A\cap B) - \bar{\mathbb{P}}_{\lambda}(A) \bar{\mathbb{P}}_{\lambda}(B) |  &\leq \widehat{\mathbb{P}}_{\lambda} (E \cap (B')^c)
\\ &\leq  \widehat{\mathbb{P}}_{\lambda}(A_{\Delta,0,t}),
\end{align}
from which we conclude the proof by applying Corollary \ref{cor covariance}.
\end{proof}

\subsection{Proof of Theorem \ref{thm local shape theorem}}
For the proof of Theorem \ref{thm local shape theorem} we follow  the idea of the proofs of Proposition $5$ and $6$ in \cite{DurrettGriffeathCPhighDim1982}. These propositions yield slightly stronger statements for the contact process on $\mathbb{Z}^d$, $d\geq1$, under the same assumption that $\lambda>\lambda_c(\mathbb{N})$. The proofs in \cite{DurrettGriffeathCPhighDim1982} use geometrical properties of $\mathbb{Z}^d$, $d\geq1$, and the argument therefore needs to be adapted in order to work for general graphs of bounded degree. (In fact, in \cite{DurrettGriffeathCPhighDim1982} only a detailed proof of the case $d=2$ is given). 

One important estimate for the proofs in \cite{DurrettGriffeathCPhighDim1982}, and for the proof of Theorem \ref{thm local shape theorem} below, is the following estimate for the contact process on $\mathbb{N}$. 

\begin{lemma}\label{lem estimates CP on N}
Consider the contact process on $\mathbb{N}$ with $\lambda>\lambda_c(\mathbb{N})$. Then there are constants $C,c,\alpha \in (0,\infty)$ such that, 
\begin{equation}
\mathbb{P}_{\lambda}^0 \left(t(x) >t  \mid \tau^0=\infty\right) \leq Ce^{-ct}, \quad  \forall \: t>0,
\end{equation}
 whenever $dist(0,x) \leq \alpha t$.
\end{lemma}

\begin{proof}
In  \cite{DurrettGriffeathCPhighDim1982} it is referred to \cite{DurretGriffeathCP1983} for a proof. That paper, however,  concerns the contact process on $\mathbb{Z}$ for which they prove the analog of Lemma \ref{lem estimates CP on N}, see Theorem 4 therein. To be precise, Theorem 4 in \cite{DurretGriffeathCP1983} states that there are constants $C,c,\alpha \in (0,\infty)$ such that
\begin{equation}
\mathbb{P}_{\lambda}^{(-\infty,0]} (r_t <\alpha t) \leq Ce^{-ct}, \quad t\geq0.
\end{equation}
Here $r_t$ denotes the position of the rightmost infected individual at time $t$. Letting $x \leq \alpha t$, we have that
\begin{align}\label{eq estimates CP on N help x}
\mathbb{P}_{\lambda}^0 \left(t(x) >t  \mid \: ^{\mathbb{N}}\tau^0=\infty\right) 
&= \mathbb{P}_{\lambda}^{0} \left(r_t > \alpha t  \mid \: ^{\mathbb{N}}\tau^0=\infty\right) 
\\&\leq \mathbb{P}_{\lambda}^{(-\infty,0]} \left( r_{t} < \alpha t  \mid \: ^{\mathbb{N}}\tau^0=\infty \right) 
\\ &\leq \mathbb{P}_{\lambda}^{(-\infty,0]} (r_{t} <\alpha t),
\end{align}
where in the last line we  have used that $\mathbb{P}_{\lambda}^{(-\infty,0]}$ is positively associated, noting that $\{ \: ^{\mathbb{N}}\tau^0=\infty \}$ and $\{ r_t <\alpha t\}^c $ are increasing events. This derivation is formally with respect to the contact process on $\mathbb{Z}$, however, by the graphical construction coupling the estimate immediately transfers to the contact process on $\mathbb{N}$.
\end{proof}

From Lemma \ref{lem DG1} and Lemma \ref{lem estimates CP on N} we derive the following proposition for the contact process on a general graph.

\begin{proposition}\label{cor survival in D}
Let $\lambda>\lambda_c(\mathbb{N})$. Then there are constants $\alpha, C,c \in (0,\infty)$ such that,
\begin{equation}\label{eq thm mixing help2}
\widehat{\mathbb{P}}_{\lambda} \left(\eta_t^x(y) \neq \eta_t^V(y) \mid \tau^x=\infty \right) \leq Ce^{-ct}
\end{equation}
 for any $x,y \in V$ satisfying $dist(x,y) \leq \alpha t$.
\end{proposition}

\begin{proof}
Let $x,y\in V$ satisfy $dist(x,y) \leq \alpha_0 t$ for some $\alpha_0 \in (0,\infty)$ to be determined. Note that the event 
$A:= \{\eta_t^x(y) \neq \eta_t^V(y), \tau^x=\infty\}$ equals, by the graphical construction in Section \ref{sec GC}, the intersection of the following three events:
\begin{align}
&\left\{ \exists \text{ active path from } V\times \{0\} \text{ to } (y,t)\right\}
\\ &\left\{\exists \text{ active path from } (x,0) \text{ to } V\times \{s\} \: \forall \: s >0\right\} 
\\ &\left\{\nexists \text{ active path from } (x,0) \text{ to } (y,t) \right\}.
\end{align}
From this we see that $A$ is contained in the union of $A_1$ and $A_2$;
\begin{align}\label{eq event dual bound}
&A_t^{(1)}:=\left\{ \hat{\eta}_t^{(y,t)} \neq \underbar{0}, \hat{\eta}_s^{(y,t)} = \underbar{0} \text{ for some } s > t  \right\} 
\\ \label{eq event dual bound2} &A_t^{(2)}:=\left\{ \tau^x=\infty, \hat{\eta}_s^{(y,t)} \neq \underbar{0} \: \forall \: s>0, \eta_s^x \hat{\eta}_{t-s}^{(y,t)}=\underbar{0} \: \forall \: s \in [0,t]\right\},
\end{align}
where we recall that $(\hat{\eta}_s^{(y,t)})_{s\geq0}$ denotes the dual process started from time $t$ with only the individual $y$ initially infected, and $\eta_s^x \hat{\eta}_{t-s}^{(y,t)}$ denotes the configuration obtain by multiplying the corresponding values of the individuals, i.e.\ $\eta_s^x \hat{\eta}_{t-s}^{(y,t)}(z) = \eta_s^x(z) \hat{\eta}_{t-s}^{(y,t)}(z)$, $z \in V$.

By Lemma \ref{lem DG1} and the self-duality property of the contact process, the probability of $A_t^{(1)}$ is bounded by $Ce^{-ct}$. The rest of the proof is devoted to showing that a similar bound holds for $A_t^{(2)}$. For this, choose the self-avoiding paths $(\gamma^{(z)})_{z \in V}$ introduced in the previous subsection to be the concatenation of a shortest path from $z$ to the graph $\Gamma^{(y)}$, say with endpoint $y_l \in \gamma^{(y)}$, and the path $(y_l,y_{l+1},\dots)$. As in the proof of Lemma \ref{lem DG1}, it follows that $\mathbb{P}_{\lambda}^x \left( T >  \epsilon t \mid \tau^x=\infty\right)$ decays exponentially in $t$, where $\epsilon \in (0,1)$ and 
\begin{equation}
T := \inf \left\{ t \geq 0 \colon \eta_t^x(z)=1 \text{ and } ^{\Gamma^{(z)}}\tau_{t}^z= \infty \text{ for some } z \in V \right\}
\end{equation}
denotes the first time that the contact process started from only $x$ infected spreads its infection within one of the subgraphs $(\Gamma^{(z)})_{z\in V}$. By a comparison with a continuous-time branching process with branching rate $D \lambda $ it furthermore follows that, for $\delta$ sufficiently large, the probability
\begin{equation}
\mathbb{P}_{\lambda}^x\left( \exists  \: z\in V, dist(z,x) \geq \delta t \colon \eta_{t}(z)=1\right)
\end{equation}
decays exponentially in $t$. Hence, we may assume that there is a $z\in V$ within distance $\delta \epsilon t$ of $x$ satisfying $\eta_T(z)=1$ and $^{\Gamma^{(z)}}\tau_{T}^z= \infty$.

By self-duality, the above argumentation also applies to the dual process $(\hat{\eta}_s^{(y,t)})$ and yields that, with a probability exponentially close to $1$ in $t$, at a time $\hat{T} \leq \epsilon t$, there is a vertex $w$ at distance at most $\delta \epsilon t$ from $y$ and satisfying that $(y,t) \rightarrow (z,t-\hat{T})$ and that this infection is spread (backwards in time) within the subgraph $\Gamma^{(w)}$. 

Finally, in order to conclude the proof, it is sufficient to control the event that the infection paths from $(z,T)$ (forward in time) intersects with the (backwards-)paths from $(w,t-\hat{T})$ in the time-interval $[T,t-\hat{T}]$. After a bit of thought, it is not difficult to see that this happens at a probability bounded from below by $1-Ce^{-ct}$ for some constants $C,c\in (0,\infty)$. Indeed, due to the particular construction of the paths $(\gamma^{(z)})_{z \in V}$, we may apply  the estimate in \eqref{eq estimates CP on N help x} to control the spread of infections from $(z,T)$ and $(w,t-\hat{T})$. Moreover, we can control the probability of intersecting infection paths by noting that, under the above analysis, we have  $dist(z,y) \leq \alpha_0 t + \delta \epsilon t$ and $dist(w,y) \leq \delta \epsilon t$ and by taking the constants $\epsilon$ and $\alpha_0$ sufficiently small. From this we conclude the proof of the proposition.
\end{proof}

\begin{proof}[Proof of Theorem \ref{thm local shape theorem}]
From the graphical construction, and by self-duality of the contact process, we have that, for any $y \in V$,
\begin{equation}\label{eq thm mixing help1}
\widehat{\mathbb{P}}_{\lambda} \left( \eta_t^{V} (y) \neq \bar{\eta}_t(y) \right) = \mathbb{P}_{\lambda}^y\left(t < \tau^y < \infty \right) \leq Ce^{-ct}, \quad \forall \: t >0,
\end{equation}
where the inequality follows by Lemma \ref{lem DG1}. Hence, by combining this with Proposition \ref{cor survival in D}, we obtain, by use of the metric inequality, that for any $x \in V$,
\begin{equation}
\widehat{\mathbb{P}}_{\lambda} \left( \eta_t^{x} (y) \neq \bar{\eta}_t(y) \mid \tau^x =\infty \right) \leq Ce^{-ct}, \quad \forall \: t>0,
\end{equation}
whenever $dist(y,x) \leq \alpha t.$  In almost exactly the same manner as in the proof of Corollary \ref{cor covariance}, we conclude from this that in fact
\begin{equation}\label{eq thm local shape help}
\widehat{\mathbb{P}}_{\lambda} \left( \eta_s^{x} (y) \neq \bar{\eta}_s(y) \text{ for some } s\in [t,t+1)  \mid \tau^x =\infty \right) \leq Ce^{-ct},
\end{equation}
since, by the graphical construction, it is sufficient to consider the times in the time interval $[t, t + 1)$ at which there is an arrow event towards $y$.
Now, let $\alpha_0 \in (0,\alpha)$ and consider an event $B \in \mathcal{F}_{\alpha_0,t}(x)$. We have that  
\begin{align}
&\left|\mathbb{P}_{\lambda}^x (B \mid \tau^x=\infty ) - \bar{\mathbb{P}}_{\lambda}(B) \right| 
\\ &\leq \widehat{\mathbb{P}} \left(\eta_s^x(y)\neq \bar{\eta}_s(y) \text{ for some }(y,s)  \in C_{\alpha_0 t}(x) \mid \tau^x =\infty\right)
\\ &\leq \sum_{l=\lfloor t \rfloor }^{\infty} \sum_{y \in \Delta_{l,\alpha_0}(x)} \widehat{\mathbb{P}}_{\lambda} \left( \eta_s^{x} (y) \neq \bar{\eta}_s(y), s\in [l,l+1)  \mid \tau^x =\infty \right),
\end{align}
where $\Delta_{l,\alpha_0}(x) := \{ y \colon dist(y,x) \leq \alpha_0 (l+1)\}$. In particular, by \eqref{eq thm local shape help} the latter sum is bounded by 
$\sum_{l=\lfloor t \rfloor }^{\infty} D^{\alpha_0 (l+1)} Ce^{-cl}$, 
which, for $\alpha_0$ close to $0$, decays exponentially in $t$. This completes the proof of the theorem.
\end{proof}

\subsection{Proof of Theorem \ref{thm phi mixing}}

From Lemma \ref{lem DG1} above, and by use of similar arguments as in the proof of  \cite{BergBethuelsenCP2018}, Theorem 1.3, we derive the following lemma.

\begin{lemma}\label{lem domination200}
Let $\Delta \in \mathcal{S}$. Then there is $\delta>0$ and $x \in V$ such that
\begin{align}\label{lem domi12345}
\inf_{A_0 \in \mathcal{F}_{\leq0}^{\Delta}} \bar{\mathbb{P}}_{\lambda} \left(\bar{\eta}_0(x)=1 \mid A_0 \right) \geq \delta.
\end{align}
\end{lemma}

\begin{proof}
Let $X=(X_i)_{i \in \mathbb{Z}}$ be the (discrete-time) process on $\{0,1\}$ given by 
\begin{equation}
X_i := \max \left\{ \bar{\eta}_t(x) \colon x \in \Delta, t \in [i,i+1) \right\}.
\end{equation}
Since $\bar{\mathbb{P}}_{\lambda}$ is dFKG  and the dFKG property is preserved under taking maximum, it follows that also the process $(X_i)$ is dFKG. Moreover, by Lemma \ref{lem DG1} applied to \cite{BergBethuelsenCP2018}, Lemma 2.5, we have that, for some $\delta>0$,
\begin{equation}\label{eq estimate X}
\bar{\mathbb{P}}_{\lambda} \left(X_i = 0, i \in [0,m-1] \right)\leq \bar{\mathbb{P}}_{\lambda} \left(\bar{\eta}_t(0) = 0, t \in [0,m] \right) \leq (1-\delta)^m.
\end{equation}
 Since $(X_i)$ is translation invariant with respect to time shifts, we can thus apply Theorem 1.2 in \cite{LiggettSteifSD2006}, which says that the estimate in \eqref{eq estimate X} is equivalent to
\begin{equation}\label{eq lem phi mixing uniform X}
\lim_{T \uparrow \infty} \bar{\mathbb{P}}_{\lambda} \left( X_0^{(n)} =1 \mid X_i^{(n)} = 0, -T < i <0  \right) \geq \delta.
\end{equation}
Note that,  by the dFKG property, the limit on the left hand side of  \eqref{eq lem phi mixing uniform X} is decreasing in $T$ and hence the limit is well defined. 
 
 Now, consider the function $h \colon V\setminus \Delta \rightarrow [0,1]$ given by
\begin{equation}
h(x):= \bar{\mathbb{P}}_{\lambda}\left(\eta_0(x)=1 \mid X_s=0, s<0\right),
\end{equation}
which, again  by the dFKG property, is well defined. We claim that $h(x)>0$ for some $x \in V\setminus \Delta$. Indeed, otherwise the measure $\bar{\mathbb{P}}_{\lambda}\left(\cdot \mid X_s=0, s<0\right)$ equals $\delta_{\underbar{0}}$ on $\mathcal{F}_0 \subset \mathcal{F}$, the $\sigma$-algebra generated by events in $V \times \{0\}$. By the Markov property of the contact process, this readily contradicts \eqref{eq lem phi mixing uniform X}. 
We thus conclude the proof of the lemma since, as a consequence of \eqref{eq dFKG SD}, we have that the righthand side of \eqref{lem domi12345} can be taken to equal the $\delta$ in \eqref{eq estimate X}.
\end{proof}

\begin{corollary}\label{lem abscon}
There exists an $\epsilon>0$ such that, for all $t$ large, we have that 
\begin{equation}
\sup_{B_t \in \mathcal{F}_{\geq t}^{\Delta} } \sup_{A_0 \in \mathcal{F}_{\leq 0}^{\Delta}} \left| \bar{\mathbb{P}}_{\lambda}(B_t \mid A_0 ) - \bar{\mathbb{P}}_{\lambda} (B_t ) \right|\leq 1-\epsilon.
\end{equation}
\end{corollary}

\begin{proof}
This follows immediately by combining Lemma \ref{lem domination200} with Theorem \ref{thm local shape theorem} together with the Markov property of the contact process.
\end{proof}

\begin{proof}[Proof of Theorem \ref{thm phi mixing}]
Firstly, the contact process on $G$ is mixing (in the ergodic-theoretic sense), that is, 
\begin{equation}\label{eq mixing ET}
\lim_{s \rightarrow \infty} \bar{\mathbb{P}}_{\lambda}(A \cap T_{-s}B) =\bar{\mathbb{P}}_{\lambda}(A) \bar{\mathbb{P}}_{\lambda}(B), \quad \forall \: A,B \in \mathcal{F}.
\end{equation}
This follows as a consequence of Proposition \ref{prop covariance}. 
In particular, \eqref{eq mixing ET} holds for the process $\xi^{(\Delta)}$, which we also recall is a stationary process. 

For a stationary process satisfying \eqref{eq mixing ET}, the mixing time $d_{\Delta}(t)$ either converges towards $0$ or it is equal to $1$ for all $t$. A detailed proof of this fact is given in \cite{BradleyStrongMixing2005}, Theorem 22.3. (\cite{BradleyStrongMixing2005} considers discrete-time processes, however, the argument extends immediately to continuous-time processes). From this and the two observations above, we thus conclude the proof of Theorem \ref{thm phi mixing} by noting that,  by Corollary \ref{lem abscon}, we have $d_{\Delta}(t)<1$ for all sufficiently large $t$.
\end{proof}

\subsection{Proof of Theorem \ref{thm phi mixing Zd}}\label{sec thm phi mixing Zd}

Important to the proof of Theorem \ref{thm phi mixing Zd}  is the following lemma, which can be seen as an extention of Lemma \ref{lem domination200} for the particular case of the contact process on $\mathbb{Z}^d, d\geq 1$.

\begin{lemma}\label{lem phi mixing Zd}
Consider the upper stationary contact process on $\mathbb{Z}^d$, $d\geq1$, with $\lambda>\lambda_c(\mathbb{Z}^d)$. Then there exists $l \in \mathbb{N}$ and $\delta>0$ such that, for all $n\in \mathbb{N}$ and $T <0$,
\begin{equation}\label{eq lem phi mixing Zd statement}
\bar{\mathbb{P}}_{\lambda} \left( \bar{\eta}_0((n+l)\cdot e_1) =1 \mid \bar{\eta}_t(x)= 0, x \in \Delta_n, t \leq [T,0] \right) \geq \delta,
\end{equation}
where $e_1=(1,0,\dots,0)$ is a unit vector.
\end{lemma}

\begin{proof} 
We first consider the case of the contact process on $\mathbb{Z}$ with $\lambda>\lambda_c(\mathbb{Z})$. 
For $n \in \mathbb{N}$, let $(Y_i^{(n)})_{i \in \mathbb{Z}}$ be the (discrete-time) process on $\{0,1\}$ defined by
\begin{equation}
Y_i^{(n)} := \max \left\{ \bar{\eta}_t(y) \colon y= j=-n,\dots, -1 \text{ and } t \in [i,i+1) \right\}, \quad i \in \mathbb{Z}.
\end{equation} 
By the same argument as in the proof of Lemma \ref{lem domination200}  we conclude that the process $(Y_i^{(n)})$ is dFKG and, moreover, that for some $\delta>0$,
\begin{equation}\label{eq lem phi mixing uniform d=1}
\lim_{T \uparrow \infty} \bar{\mathbb{P}}_{\lambda} \left( Y_0^{(n)} =1 \mid Y_i^{(n)} = 0, -T < i <0  \right) \geq \delta. 
\end{equation}
Observe that this estimate holds irrespectively of $n$.

Further, consider the function $f \colon \mathbb{N} \rightarrow [0,1]$ given by
\begin{equation}
f(x):= \lim_{n\rightarrow \infty} \lim_{T \uparrow \infty} \bar{\mathbb{P}}_{\lambda} \left( \bar{\eta}_0(x) =1 \mid Y_i^{(n)} = 0, -T < i <0  \right)
\end{equation}
which, again  by the dFKG property, is well defined. We claim that $f(x)>0$ for some $x \in \mathbb{N}$. For a contradiction, assume that this is not the case. Hence, for any fixed $M \in \mathbb{N}$ and $\epsilon>0$, we may chose $N$ so large so that, for all $n\geq N$, we have that
\begin{equation}\label{eq estimate last night}
 \lim_{T \uparrow \infty} \bar{\mathbb{P}}_{\lambda} \left( \bar{\eta}_0(x) =1 \mid Y_i^{(n)} = 0, -T < i <0  \right) < \epsilon
\end{equation}
for all $x \in[1,M] \cup [-(M+n),n]$, where we also make use of the translation invariance of the contact process. The estimate \eqref{eq estimate last night}, however, is easily seen to  contradict \eqref{eq lem phi mixing uniform d=1} by choosing  $\epsilon$ and $M$ appropriately. We thus conclude that $f(x)>0$ for some $x \in \mathbb{N}$ and from this, by translation invariance and the dFKG property, we conclude \eqref{eq lem phi mixing Zd statement} for the case $\mathbb{Z}$ and $\lambda>\lambda_c(\mathbb{Z})$.

We next consider the case of the contact process on $\mathbb{Z}^d$ with $\lambda>\lambda_c(\mathbb{Z}^d)$ for which we need to adapt the above proof slightly. 
For $n \in \mathbb{N}$, let $(Z_i^{(n)})$ be the (discrete-time) process on $\{0,1\}^{\mathbb{Z}^{d-1}}$ defined by
\begin{equation}
Z_i^{(n)}(x) := \max \left\{ \bar{\eta}_t(y) \colon y=(j,x) \in \mathbb{Z}^d, j\in [-n,0], t \in [i,i+1)  \right\}.
\end{equation} 
Again this process is dFKG. Moreover, it holds that, for some $\delta>0$,
\begin{equation}
\bar{\mathbb{P}}_{\lambda} \left(Z_i^{(n)}(x) = 0 \text{ for all } (x,i) \in [0, m]^{d-1} \times [0,m] \right)\leq (1-\delta)^{m^d},
\end{equation}
as shown in \cite{BergBethuelsenCP2018}, Theorem 1.5. Thus, by applying  Theorem 4.1 in \cite{LiggettSteifSD2006} (see also Lemma 2.4 in \cite{BergBethuelsenCP2018}), we conclude that
\begin{equation}\label{eq lem phi mixing uniform}
\bar{\mathbb{P}}_{\lambda} \left( Z_0^{(n)}(o) =1 \mid Z_i^{(n)}(x) = 0 \text{ for all } x \in \mathbb{Z}^{d-1} \text{ and } i <0 \right) \geq \delta,
\end{equation}
where $o \in \mathbb{Z}^{d-1}$ denotes the origin. Next, let $g \colon \mathbb{N} \rightarrow [0,1]$ be given by
\begin{equation}
g(m):= \lim_{n\rightarrow \infty} \lim_{T \uparrow \infty} \bar{\mathbb{P}}_{\lambda} \left( \exists  x\in\Delta_m \colon \bar{\eta}_0(x) =1 \mid Z_i^{(n)}(x) = 0,  i \in (-T,0), x \in \mathbb{Z}^{d-1}  \right)
\end{equation}
which, again  by the dFKG property, is well defined. We claim that $g(m)>0$ for some $m \in \mathbb{N}$. Indeed, the opposite contradicts \eqref{eq lem phi mixing uniform}, as can be seen by arguing similarly as in the $d=1$ case. By this claim, and using translation invariant and the dFKG property, we conclude \eqref{eq lem phi mixing Zd statement} also in this case and hence the proof of Lemma \ref{lem phi mixing Zd} is complete.\end{proof}

\begin{remark}
The above proof strongly relies on symmetries of the contact process on $\mathbb{Z}^d$ and cannot immediately be extended to general graphs.  These problems, however, can be  circumvented for discrete-time analogs of the contact process for which a version of  \eqref{eq lem phi mixing Zd statement} holds, replacing $\Delta_n$ by any simply connected set $\Delta$. Indeed, by using that infections spread only to neighbouring individuals and that it is a process in discrete-time, by following the argument for the contact process on $\mathbb{Z}$ above, it is not difficult to see that
\begin{equation}\label{eq discrete uniform bound}
\bar{\mathbb{P}}_{\lambda} \left( \bar{\eta}_0(x) =1, x \in \partial \Delta \mid \bar{\eta}_t(x)= 0, x \in \Delta, t \leq [T,0] \cap \mathbb{Z} \right) \geq \delta,
\end{equation}
where $\partial \Delta = \{ x \notin \Delta \colon dist(x,y)=1 \text{ for some } y \in \Delta\}$ is the outer boundary of $\Delta$. 
\end{remark}

Before presenting the proof of Theorem \ref{thm phi mixing Zd}, we first recall a couple of large deviation estimates from \cite{GaretMarchandLDP2014}, see Theorem 1 and 4 therein, for the contact process on $\mathbb{Z}^d$. (\cite{GaretMarchandLDP2014} in fact considered a generalisation of the contact process, evolving in a \emph{random environment}). For this, denote by
\begin{equation}
t'(x):= \inf \{ T \geq 0 \colon \forall t \geq T, \eta_t^o(x) = \eta_t^{\mathbb{Z}^d}(x) \}, \quad x \in \mathbb{Z}^d,
\end{equation}
and recall the definitions of $t(x)$ and $\beta$; see Equations \eqref{eq t(x)} and \eqref{eq beta}.

\begin{lemma}\label{lem Garet Marchand}
Consider the contact process on $\mathbb{Z}^d$, $d\geq 1$, with $\lambda>\lambda_c(\mathbb{Z}^d)$. Then there are constant $C,c\in(0,\infty)$ such that for any $\epsilon>0$ we have that
\begin{align}
&\mathbb{P}_{\lambda}^o \left( t'(n \cdot e_1) \geq \beta n (1+\epsilon) \mid \tau^o=\infty \right) \leq Ce^{-c n}
\\&\mathbb{P}_{\lambda}^o \left( t(x) \leq \beta n (1-\epsilon), dist(x,o) \geq n \mid \tau^o=\infty \right) \leq Ce^{-c n}
\end{align}

\end{lemma}

\begin{proof}[Proof of Theorem \ref{thm phi mixing Zd}]
Let $\bar{\eta}$ be the upper stationary contact process on $\mathbb{Z}^d$, $d\geq1$, with $\lambda>\lambda_c(\mathbb{Z}^d)$. Then there are constants $C,c \in (0,\infty)$ such that, for any $B \in \mathcal{F}_{\geq t}^{\Delta_n}$, 
\begin{equation}\label{eq eq seq seq}
\left|\bar{\mathbb{P}}_{\lambda}(B) - \mathbb{P}_{\lambda}^{\mathbb{Z}^d}( B) \right| \leq Cn^d e^{-ct}, \quad \forall \: t\geq0.
\end{equation} 
This follows by self-duality and a standard union bound together with the fact that the contact process started from all individuals infected converges exponentially fast to $\bar{\nu}_{\lambda}$, that is, Theorem 1.2.30 in \cite{LiggettSIS1999}. 

 In order to conclude the first statement of the theorem, by the dFKG property (in particular, \eqref{eq dFKG SD}), it is thus sufficient to show that for each $\epsilon>0$ there are constant $C,c\in (0,\infty)$ such that, for every $B \in \mathcal{F}_{\geq t}^{\Delta_n}$ we have that
\begin{equation}\label{eq last thm HELP}
|\bar{\mathbb{P}}_{\lambda}(B) - \bar{\mathbb{P}}_{\lambda}(B \mid A_0)|\leq C n^d e^{-c t}, \quad \forall \: t\geq \beta n (1+\epsilon),
\end{equation} 
where $A_0$ here denotes the event $\{  \bar{\eta}_t(x) =0 \: \forall \: (x,t) \in \Lambda \times (-\infty,0]\}$.

To this end, denote by $\widehat{\mathbb{P}}_{\lambda}$ the coupling of $(\eta_t)_{t\geq0}$ and $(\eta_t^{\mu})_{t\geq0}$ via the graphical construction, where $\mu(\cdot) =  \mathbb{P}(\eta_0\in \cdot \mid A_0)$ is a measure on $(\Omega, \mathcal{F}_0)$  and $\eta_t^{\mu}$ is the contact process started with a configuration drawn according to $\mu$ at time $0$. We have that
\begin{align}\label{eq eq beq beq}
\left| \bar{\mathbb{P}}_{\lambda}(B) - \mathbb{P}_{\lambda}^{\mu}(B) \right| 
\leq  \sum_{x \in \Delta_n} \widehat{\mathbb{P}}_{\lambda} \left( \bar{\eta} \neq \eta^{\mu} \text{ on } \{x\} \times [t, \infty) \right).
 \end{align}
 Let $\epsilon>0$ and consider for each $x\in \Delta_n$ the event 
 \begin{equation}
 M_x := \inf \{ dist(y,x) \colon \eta_0^{\mu}(y)=1 \text{ and } \tau^y = \infty \}.
 \end{equation} 
In order to control the terms within the sum of \eqref{eq eq beq beq}, we use that
 \begin{align}
 &\widehat{\mathbb{P}}_{\lambda} \left( \bar{\eta} \neq \eta^{\mu} \text{ on } \{x\} \times [t, \infty) \right)
 \\ \label{eq eq eq eq} \leq &\widehat{\mathbb{P}}_{\lambda} \left( \{M_x>\epsilon t\} \right) + \widehat{\mathbb{P}}_{\lambda} \left( \{\bar{\eta} \neq \eta^{\mu} \text{ on } \{x\} \times [t, \infty) \} \cap \{M_x \leq \epsilon t\} \right).
 \end{align}

The first term on the righthand side of \eqref{eq eq eq eq} decays exponentially in $\delta t$. Indeed, by Lemma \ref{lem phi mixing Zd} above there is an $l\in \mathbb{N}$ and $\delta>0$ such that the measure $\mu$ stochastically dominates a Bernoulli product measure  on the subset $\{ y \in \mathbb{Z}^d \colon y=x+ e_x \cdot(n \pm rl), r \in \mathbb{N} \} $ with density $\delta$. Hence, by standard large deviation estimates for Bernoulli measures combined with Theorem 1.2.30 in \cite{LiggettSIS1999}, we conclude the exponential decay of $\mathbb{P}_{\lambda}^{\mu} \left( M_x>\epsilon t \right)$ in $t \geq0$. Next, the second term on the righthand side of \eqref{eq eq eq eq} decays exponentially whenever $t \geq \beta n (1+\epsilon)$, as a consequence of the first inequality in Lemma \ref{lem Garet Marchand}. From these two bounds, which hold uniformly in $x\in \Delta_n$, we conclude that \eqref{eq eq eq eq} decays exponentially in $t$. In particular, we have that 
\begin{equation}
\eqref{eq eq beq beq} 
\leq \sum_{x \in \Delta_n}  Ce^{-ct} \leq Cn^de^{-ct}, \quad \forall \: t\geq \beta n (1+\epsilon),
 \end{equation}
which together with \eqref{eq eq seq seq} yields the statement of \eqref{eq phi mixing Zd upper}.

For the second part, that is \eqref{eq cutoff}, we first note that \eqref{eq phi mixing Zd upper} implies that, for any $\epsilon>0$, we have $d_{\Delta_n}(\beta n (1+\epsilon)) \rightarrow 0$ as $n \rightarrow \infty$. In order to prove a lower bound, denote by $(\eta_t^{\Delta_n^c})$ the contact process started from the configuration where all individual inside $\Delta_n$ are healthy and all the other individuals are infected. Further, let $B_r:= \{\eta_0(x) =1 \text{ for some } x \in \Delta_r\}$. Then, for any $r \in \mathbb{N}$ and $\epsilon>0$, 
\begin{equation}\label{eq cutoff 1}
 \mathbb{P}_{\lambda}^{\Delta_n^c}(\eta_{\beta n(1-\epsilon)}^{\Delta_n^c} \in B_r ) \leq C r^d e^{-c(n-r)}, \quad \forall \: n \geq 0,
 \end{equation}
 which for fixed $r$ goes to $0$ as $n\rightarrow \infty$. Indeed, this bound follows as a consequence of the second inequality in Lemma \ref{lem Garet Marchand} above and by the self-duality of the contact process. From this we obtain the second limit in \eqref{eq cutoff} since $\bar{\mathbb{P}}_{\lambda}(B_r)\rightarrow 1$ as $r \rightarrow \infty$, and by this we conclude the proof.
\end{proof}

\begin{remark}\label{rem phi mixing Zd}
Some of the arguments in the proof of Theorem \ref{thm phi mixing Zd} can surely be extended to other graphs. In particular, assuming that $\lambda> \lambda_c(\mathbb{N})$, one can instead of Lemma \ref{lem Garet Marchand} use the estimate on the spread of an infection obtained in Theorem \ref{thm local shape theorem}. An important step in the proof of \eqref{eq phi mixing Zd upper} above is, however, the use of Lemma \ref{lem phi mixing Zd}.  On general graphs, we do not know how to obtain such an uniform estimate.

Interestingly, these problems can be completely  circumvented for discrete-time analogs of the contact process. By applying \eqref{eq discrete uniform bound} in place of Lemma \ref{lem phi mixing Zd}, the first part of the proof of Theorem \ref{thm phi mixing Zd} goes through without difficulty, using a  discrete-time version of Theorem \ref{thm local shape theorem} for the large deviation estimates in the final estimate.
\end{remark}

\section{Discussion and open questions}\label{sec discussion}
In this last section we mention some potential extensions of the theory so far presented and discuss some in our opinion intriguing open questions.

\begin{enumerate}
\item In Section \ref{sec CP} we introduced the notions weak and strong survival. Other critical parameter values have also been considered in the literature, see for instance \cite{PemantleCP1992}. We next define yet another critical parameter value which to our knowledge has not appeared in the literature before. 
Given a  connected and countable-infinite graph $G=(V,E)$, define
\begin{equation} 
\lambda_T(G) := \sup \left\{\lambda>0 \colon \rho(\lambda,x) =0 \text{ for all } x\in V  \right\},
\end{equation}
where, for $x \in V$, we set 
\begin{equation}
\rho_{\lambda}(x):= \lim_{t \rightarrow \infty} \frac{1}{t} \log \left( \bar{\mathbb{P}}_{\lambda} ( \eta_s(x)=0, s\in [0,t) \right), \quad x\in V.
\end{equation} 
Note that both $\lambda_T(G)$ and $\rho_{\lambda}(x)$  are well defined due to monotonicity of the contact process (both in $G$ and $\lambda$) and by the large deviation principle in Theorem \ref{thm LDP}.

$\quad$It is not difficult to see that $\lambda_T(G)\in (0,\infty)$. An upper bound  on $\lambda_T(G)$ holds since its value is larger than that of weak survival. Indeed, $\bar{\nu}_{\lambda}$ has to be non-trivial in order for $\rho(\lambda)$ to be positive.  Further, $\lambda_c(\mathbb{N})$ yields a lower bound on $\lambda_T(G)$ as follows from Lemma \ref{lem DG1} applied to Lemma 2.5 in \cite{BergBethuelsenCP2018}. In fact, for many graphs, such as $\mathbb{Z}^d$, $d\geq1$, the upper and lower bound can be shown to match and hence $\lambda_T(G)= \lambda_c(G)$ in these cases. 

$\quad$\textbf{Question $1$:} How does $\lambda_T(G)$ relate to the critical values defined through the notions weak survival and strong survival? In particular, does there exist graphs $G$ for which $\lambda_T(G) > \lambda_c(G)$?

$\quad$\textbf{Question $2$}a: Does the contact process on a graph $G$ always converge exponentially fast (in the weak sense) to its equilibrium state when started from all individuals infected when $\lambda>\lambda_c(G)$?

$\quad$\textbf{Question $2$}b: Equivalently (by duality), for the contact process on a graph $G$, does there exists constants $C,c>0$ such that, for all $x \in V$, we have that $\mathbb{P}(s \leq \tau^x< \infty) \leq Ce^{-cs}$ for all $s>0$ whenever $\lambda>\lambda_c(G)$?

\begin{remark}
Any values of $\lambda$ fulfilling the requirements of Question $2$ also satisfy that $\rho(\lambda)>0$ as can been seen by a classical restart argument. Hence, a positive answer to Question $2$ would yield a negative answer to Question $1$. 
\end{remark}

Based on Corollary \ref{cor LDP} a natural followup question is whether large deviation properties hold as soon as $\lambda>\lambda_T$.

$\quad$\textbf{Question $3$:} Does the contact process on a graph $G$ projected onto a finite subset always satisfy large deviation estimates similar to those seen in Corollary \ref{cor LDP}, when $\lambda> \lambda_c(G)$ or $\lambda>\lambda_T(G)$?

\item  As discussed in Section \ref{sec thm phi mixing Zd},  the estimates on the mixing times in Theorem \ref{thm phi mixing Zd} can also be proven for certain other graphs and in more generality for discrete-time analogs of the contact process. Motivated by this the following question seems natural.
 
$\quad$\textbf{Question} 4: 
For each $\lambda>\lambda_c(\mathbb{N})$ and $\epsilon>0$, does there exist constants $\alpha, C,c \in (0,\infty)$ such that the contact process on a graph $G$ projected onto any $\Delta \in \mathcal{S}$ satisfies that $d_{\Delta}(t) \leq C|\Delta| e^{-c t}$ whenever $t \geq \alpha \text{diam}(\Delta) (1+\epsilon)$? If yes, does it also hold for all $\lambda> \lambda_c(G)$ or $\lambda>\lambda_T(G)$?

$\quad$A natural followup questions to Question $4$ is whether the contact process in the regimes considered therein also exhibits the cutoff phenomena, as obtained in Theorem \ref{thm phi mixing Zd} only for the special case of the contact process on $\mathbb{Z}^d$.
It would also be interesting to study the dependency on $\Delta \in \mathcal{S}$ for other statistics of $(\xi^{\Delta})$. One case treated in the literature to date is the asymptotic behaviour of the occurrence times of rare events, which has been studied for the contact process on $\mathbb{Z}$ in \cite{LebowitzSchonmannLDP1987} and \cite{GalvesMarinelliOlvieriCP1989}. We postulate that these works can be extended to general countable-infinite graphs by similar methods to those developed in this paper.

\item We find it somewhat surprising that projections of the  contact process satisfy such a strong ``loss of memory''-property as seen by Theorems \ref{thm phi mixing} and  \ref{thm phi mixing Zd}, and it would be interesting to check whether other types of interacting particle systems may satisfy the same kind of mixing properties.  The main technical tools used in this paper are the dFKG property, self-duality and estimates on the rate of convergences towards the equilibrium state. Presumably, many of our results can be extended to a larger class of attractive spin-flip systems for which such properties are known to hold, see \cite{BezuidenhoutGrayCAS1994} and \cite{LiggettCPcor2006}. Another interesting model to consider is the voter model as studied e.g.\ in \cite{BramsonCoxGriffeath1988}. 

$\quad$\textbf{Question} 5: Is the occupation time process of the Voter model on $\mathbb{Z}^d$, as considered in \cite{BramsonCoxGriffeath1988}, a $\phi$-mixing process for any dimension $d\geq3$? 

\end{enumerate}

\subsection*{Acknowledgement}
The author thanks the IMS and the organisers of the program \emph{Genealogies of Interacting Particle Systems}, July-August 2017, Institute for Mathematical Sciences, Singapore, where this work was initiated. He thanks Daniel Valesin for stimulating discussions at an early stage of this project and the referee for valuable comments. Financial support from the DFG, project GA582/7-2, is greatly acknowledged.

\end{document}